\pdfoutput=1
\RequirePackage{ifpdf}
\ifpdf 
\documentclass[pdftex]{sigma}
\else
\documentclass{sigma}
\fi

\numberwithin{equation}{section}

\newtheorem{Theorem}{Theorem}[section]
\newtheorem{Lemma}[Theorem]{Lemma}

\begin{document}

\allowdisplaybreaks

\newcommand{\arXivNumber}{1711.01724}

\renewcommand{\PaperNumber}{010}

\FirstPageHeading

\ShortArticleName{Some Remarks on the Total CR $Q$ and $Q^\prime$-Curvatures}

\ArticleName{Some Remarks on the Total CR $\boldsymbol{Q}$ and $\boldsymbol{Q^\prime}$-Curvatures}

\Author{Taiji MARUGAME}

\AuthorNameForHeading{T.~Marugame}

\Address{Institute of Mathematics, Academia Sinica, Astronomy-Mathematics Building, \\
No.~1, Sec.~4, Roosevelt Road, Taipei 10617, Taiwan}
\Email{\href{mailto:marugame@gate.sinica.edu.tw}{marugame@gate.sinica.edu.tw}}

\ArticleDates{Received November 09, 2017, in f\/inal form February 12, 2018; Published online February 14, 2018}

\Abstract{We prove that the total CR $Q$-curvature vanishes for any compact strictly pseudoconvex CR manifold. We also prove the formal self-adjointness of the $P^\prime$-operator and the CR invariance of the total $Q^\prime$-curvature for any pseudo-Einstein manifold without the assumption that it bounds a Stein manifold.}


\Classification{32V05; 52T15}

\section{Introduction}
The $Q$-curvature, which was introduced by T.~Branson~\cite{Branson}, is a fundamental curvature quantity on even dimensional conformal manifolds. It satisf\/ies a simple conformal transformation formula and its integral is shown to be a global conformal invariant. The ambient metric construction of the $Q$-curvature \cite{Fefferman/Hirachi} also works for a CR manifold $M$ of dimension $2n+1$, and we can def\/ine the CR $Q$-curvature, which we denote by $Q$. The CR $Q$-curvature is a CR density of weight $-n-1$ def\/ined for a f\/ixed contact form $\theta$\, and is expressed in terms of the associated pseudo-hermitian structure. If we take another contact form $\widehat\theta=e^\Upsilon\theta$, $\Upsilon\in C^\infty(M)$, it transforms as
\begin{gather*}
\widehat Q=Q+P\Upsilon,
\end{gather*}
where $P$ is a CR invariant linear dif\/ferential operator, called the (critical) CR GJMS operator. Since $P$ is formally self-adjoint and kills constant functions, the integral
\begin{gather*}
\overline{Q}=\int_M Q,
\end{gather*}
called the total CR $Q$-curvature, is invariant under rescaling of the contact form and gives a~global CR invariant of $M$. However, it follows readily from the def\/inition of the CR $Q$-curvature that $Q$ vanishes identically for an important class of contact forms, namely the pseudo-Einstein contact forms. Since the boundary of a Stein manifold admits a pseudo-Einstein contact form~\cite{Cao/Chang}, the CR invariant $\overline{Q}$ vanishes for such a CR manifold. Moreover, it has been shown that on a~Sasakian manifold the CR $Q$-curvature is expressed as a~divergence~\cite{Alexakis/Hirachi}, and hence $\overline{Q}$ also vanishes in this case. Thus, it is reasonable to conjecture that the total CR $Q$-curvature vanishes for any CR manifold, and our f\/irst result is the conf\/irmation of this conjecture:

\begin{Theorem}\label{van-total-CR-Q}
Let $M$ be a compact strictly pseudoconvex CR manifold. Then the total CR $Q$-curvature of $M$ vanishes: $\overline{Q}=0$.
\end{Theorem}

For three dimensional CR manifolds, Theorem \ref{van-total-CR-Q} follows from the explicit formula of the CR $Q$-curvature; see \cite{Fefferman/Hirachi}. In higher dimensions, we make use of the fact that a compact strictly pseudoconvex CR manifold $M$ of dimension greater than three can be realized as the boundary of a complex variety with at most isolated singularities \cite{Monvel, Harvey/Lawson:1975,
Harvey/Lawson:1977}. By resolution of singularities, we can realize $M$ as the boundary of a complex manifold $X$ which may not be Stein. In this setting, the total CR $Q$-curvature is characterized as the logarithmic coef\/f\/icient of the volume expansion of the asymptotically K\"ahler--Einstein metric on $X$ \cite{Seshadri}. By a simple argument using Stokes' theorem, we prove that there is no logarithmic term in the expansion.

Although the vanishing of $\overline{Q}$ is disappointing, there is an alternative $Q$-like object on a CR manifold which admits pseudo-Einstein contact forms. Generalizing the operator of Branson--Fontana--Morpurgo \cite{Branson/Fontana/Morpurgo} on the CR sphere, Case--Yang~\cite{Case/Yang} (in dimension three) and Hirachi \cite{Hirachi} (in general dimensions) introduced the $P^\prime$-operator and the $Q^\prime$-curvature for pseudo-Einstein CR manifolds. Let us denote the set of pseudo-Einstein contact forms by $\mathcal{PE}$ and the space of CR pluriharmonic functions by $\mathcal{P}$. Two pseudo-Einstein contact forms $\theta, \widehat\theta\in\mathcal{PE}$ are related by $\widehat\theta=e^\Upsilon \theta$ for some $\Upsilon\in\mathcal{P}$. For a f\/ixed $\theta\in\mathcal{PE}$, the $P^\prime$-operator is def\/ined to be a linear dif\/ferential operator on $\mathcal{P}$ which kills constant functions and satisf\/ies the transformation formula
\begin{gather*}
\widehat P^\prime f=P^\prime f+P(f\Upsilon)
\end{gather*}
under the rescaling $\widehat\theta=e^{\Upsilon}\theta$. The $Q^\prime$-curvature is a CR density of weight $-n-1$ def\/ined for $\theta\in\mathcal{PE}$, and satisf\/ies
\begin{gather*}
\widehat Q^\prime=Q^\prime+2P^\prime \Upsilon+P\big(\Upsilon^2\big)
\end{gather*}
for the rescaling. Thus, if $P^\prime$ is formally self-adjoint on $\mathcal{P}$, the total $Q^\prime$-curvature
\begin{gather*}
\overline{Q}^\prime=\int_M Q^\prime
\end{gather*}
gives a CR invariant of $M$. In dimension three and f\/ive, the formal self-adjointness of $P^\prime$ follows from the explicit formulas \cite{Case/Gover, Case/Yang}. In higher dimensions, Hirachi \cite[Theorem~4.5]{Hirachi} proved the formal self-adjointness under the assumption that $M$ is the boundary of a Stein manifold $X$; in the proof he used Green's formula for the asymptotically K\"ahler--Einstein metric $g$ on $X$, and the global K\"ahlerness of $g$ was needed to assure that a pluriharmonic function is harmonic with respect to $g$. In this paper, we slightly modify his proof and prove the self-adjointness of $P^\prime$ for general pseudo-Einstein manifolds:

\begin{Theorem}\label{self-adjoint} Let $M$ be a compact strictly pseudoconvex CR manifold. Then the $P^\prime$-operator for a pseudo-Einstein contact form satisfies
\begin{gather*}
\int_M \big(f_1P^\prime f_2-f_2 P^\prime f_1\big)=0
\end{gather*}
for any $f_1, f_2\in\mathcal{P}$.
\end{Theorem}
Consequently, the CR invariance of $\overline{Q}^\prime$ holds for any CR manifold which admits a pseudo-Einstein contact form:
\begin{Theorem}\label{CR-invariance}
Let $M$ be a compact strictly pseudoconvex CR manifold which admits a pseudo-Einstein contact form. Then the total $Q^\prime$-curvature is independent of the choice of $\theta\in\mathcal{PE}$.
\end{Theorem}

We note that $\overline{Q}^\prime$ is a nontrivial CR invariant since it has a nontrivial variational formula; see \cite{Hirachi/Marugame/Matsumoto}. We also give an alternative proof of Theorem \ref{CR-invariance} by using the characterization \cite[Theorem~5.6]{Hirachi} of
$\overline{Q}^\prime$ as the logarithmic coef\/f\/icient in the expansion of some integral over a complex manifold with boundary $M$.

\section{Proof of Theorem \ref{van-total-CR-Q}}
We brief\/ly review the ambient metric construction of the CR $Q$-curvature; we refer the reader to~\cite{Fefferman/Hirachi, Hirachi, Hirachi/Marugame/Matsumoto} for detail.

Let $\overline X$ be an $(n+1)$-dimensional complex manifold with strictly pseudoconvex CR boundary~$M$, and let $r\in C^\infty(\overline{X})$ be a boundary def\/ining function which is positive in the interior~$X$. The restriction of the canonical bundle $K_{\overline X}$ to $M$ is naturally isomorphic to the CR canonical bundle $K_M:=\wedge^{n+1}(T^{0, 1}M)^\perp
\subset\wedge^{n+1}(\mathbb{C} T^\ast M)$. We def\/ine the {\it ambient space} by $\widetilde X=K_{\overline X}\setminus\{0\}$, and set
$\mathcal{N}=K_M\setminus\{0\}\cong\widetilde X|_M$. The density bundles over $\overline X$ and $M$ are def\/ined by
\begin{gather*}
\widetilde{\mathcal{E}}(w)=\big(K_{\overline X}\otimes\overline{K}_{\overline X}\big)^{-w/(n+2)}, \qquad
\mathcal{E}(w)=\big(K_M\otimes\overline{K}_M\big)^{-w/(n+2)}\cong\widetilde{\mathcal{E}}(w)|_M
\end{gather*}
for each $w\in\mathbb{R}$. We call $\mathcal{E}(w)$ the {\it CR density bundle} of weight~$w$. The space of sections of $\widetilde{\mathcal{E}}(w)$ and $\mathcal{E}(w)$ are also denoted by the same symbols. We def\/ine a $\mathbb{C}^\ast$-action on $\widetilde X$ by $\delta_\lambda u=
\lambda^{n+2}u$ for $\lambda\in\mathbb{C}^\ast$ and $u\in\widetilde X$. Then a section of $\widetilde{\mathcal{E}}(w)$ can be identif\/ied with a function on $\widetilde X$ which is homogeneous with respect to this action:
\begin{gather*}
\widetilde{\mathcal{E}}(w)\cong \big\{f\in C^\infty\big(\widetilde X\big)\, |\, \delta_\lambda^\ast f=|\lambda|^{2w}f\ {\rm for}\ \lambda\in\mathbb{C}^\ast\big\}.
\end{gather*}
Similarly, sections of $\mathcal{E}(w)$ are identif\/ied with homogeneous functions on $\mathcal{N}$.

Let $\boldsymbol\rho\in\widetilde{\mathcal{E}}(1)$ be a density on $\overline X$ and $(z^1, \dots, z^{n+1})$ local holomorphic coordinates. We set $\rho=|{\rm d}z^1\wedge\cdots\wedge {\rm d}z^{n+1}|^{2/(n+2)}\boldsymbol\rho\in\widetilde{\mathcal{E}}(0)$ and def\/ine
\begin{gather*}
\mathcal{J}[\boldsymbol\rho]:=(-1)^{n+1}\det
\begin{pmatrix}
\rho & \partial_{z^{\overline j}}\rho \\
\partial_{z^i}\rho & \partial_{z^i}\partial_{z^{\overline j}}\rho
\end{pmatrix}.
\end{gather*}
Since $\mathcal{J}[\boldsymbol\rho]$ is invariant under changes of holomorphic coordinates, $\mathcal{J}$ def\/ines a global dif\/ferential operator, called the {\it Monge--Amp\`ere operator}. Fef\/ferman \cite{Fefferman} showed that there exists $\boldsymbol\rho\in\widetilde{\mathcal{E}}(1)$ unique modulo
$O(r^{n+3})$ which satisf\/ies $\mathcal{J}[\boldsymbol\rho]=1+O(r^{n+2})$ and is a def\/ining function of $\mathcal{N}$.
We f\/ix such a $\boldsymbol\rho$ and def\/ine the {\it ambient metric} $\widetilde g$ by the Lorentz--K\"ahler metric on a neighborhood of $\mathcal{N}$ in $\widetilde X$ which has the K\"ahler form $-i\partial\overline{\partial}\boldsymbol\rho$.

Recall that there exists a canonical weighted contact form $\boldsymbol{\theta}\in \Gamma(T^\ast M\otimes\mathcal{E}(1))$ on $M$, and the choice of a contact form $\theta$ is equivalent to the choice of a~positive section $\tau\in\mathcal{E}(1)$, called a~{\it CR scale};
they are related by the equation $\boldsymbol{\theta}=\tau \theta$. For a CR scale $\tau\in\mathcal{E}(1)$, we def\/ine the CR $Q$-curvature by
\begin{gather*}
Q=\widetilde{\Delta}^{n+1}\log\widetilde\tau\, |_{\mathcal{N}}\in\mathcal{E}(-n-1),
\end{gather*}
where $\widetilde\Delta=-\widetilde\nabla_I\widetilde\nabla^I$ is the K\"ahler Laplacian of $\widetilde g$ and $\widetilde\tau\in\widetilde{\mathcal{E}}(1)$ is an arbitrary extension of $\tau$. It can be shown that $Q$ is independent of the choice of an extension of $\tau$, and the total CR $Q$-curvature
$\overline{Q}$ is invariant by rescaling of $\tau$.

The total CR $Q$-curvature has a characterization in terms of a complete metric on $X$. We note that the $(1, 1)$-form $-i\partial\overline{\partial}\log \boldsymbol\rho$ descends to a K\"ahler form on $X$ near the boundary. We extend this K\"ahler metric to a hermitian metric $g$ on $X$. The K\"ahler Laplacian $\Delta=-g^{i\overline j}\nabla_i\nabla_{\overline j}$ of $g$ is related to $\widetilde\Delta$ by the equation
\begin{gather}\label{laplacians}
\boldsymbol\rho\widetilde\Delta f=\Delta f, \qquad f\in\widetilde{\mathcal{E}}(0)
\end{gather}
near $\mathcal{N}$ in $\widetilde X\setminus\mathcal{N}$. In the right-hand side, we have regarded $f$ as a function on $X$.

For any contact form $\theta$ on $M$, there exists a boundary def\/ining function $\rho$ such that
\begin{gather}\label{bdf}
\vartheta|_{TM}=\theta, \qquad |\partial\log\rho|_g=1\ {\rm near}\ M\ {\rm in}\ X,
\end{gather}
where $\vartheta:=\operatorname{Re} (i\partial\rho)$ (\cite[Lemma 3.1]{Seshadri}). Let $\xi$ be the $(1, 0)$-vector f\/iled on $\overline X$ near $M$ characterized by
\begin{gather*}
\xi\rho=1, \qquad \xi\perp_g \mathcal{H},
\end{gather*}
where $\mathcal{H}:=\operatorname{Ker}\partial \rho\subset T^{1,0}\overline X$. Then, $N:=\operatorname{Re}\xi$ is smooth up to the boundary and satisf\/ies $N\rho=1$, $\vartheta(N)=0$. Moreover,
$\nu:=\rho N$ is $(\sqrt 2)^{-1}$ times the unit outward normal vector f\/iled along the level sets of $\rho$. By Green's formula, for any function $f$ on $X$ we have
\begin{gather}\label{Green1}
\int_{\rho>\epsilon} \Delta f\operatorname{vol}_g= \int_{\rho=\epsilon} \nu f\, \nu\lrcorner\operatorname{vol}_g.
\end{gather}
Since the Monge--Amp\`ere equation implies that $g$ satisf\/ies
\begin{gather*}
\operatorname{vol}_g=-(n!)^{-1}(1+O(\rho))\rho^{-n-2}{\rm d}\rho\wedge\vartheta\wedge({\rm d}\vartheta)^n,
\end{gather*}
the formula \eqref{Green1} is rewritten as
\begin{gather}\label{Green2}
\int_{\rho>\epsilon} \Delta f\operatorname{vol}_g=-(n!)^{-1}\int_{\rho=\epsilon} N f\cdot(1+O(\epsilon))\epsilon^{-n} \vartheta\wedge ({\rm d}\vartheta)^n.
\end{gather}
With this formula, we prove the following characterization of $\overline Q$.

\begin{Lemma}[{\cite[Proposition~A.3]{Seshadri}}]\label{Q-lp}
For an arbitrary defining function $\rho$, we have
\begin{gather*}
\operatorname{lp}\int_{\rho>\epsilon} \operatorname{vol}_g=\frac{(-1)^n}{(n!)^2(n+1)!}\overline{Q},
\end{gather*}
where $\operatorname{lp}$ denotes the coefficient of $\log\epsilon$ in the asymptotic expansion in~$\epsilon$.
\end{Lemma}
\begin{proof}
Since the coef\/f\/icient of $\log\epsilon$ in the volume expansion is independent of the choice of~$\rho$ \cite[Proposition~4.1]{Seshadri}, we may assume that $\rho$ satisf\/ies~\eqref{bdf} for a~f\/ixed contact~$\theta$ on~$M$. We take $\widetilde\tau\in\widetilde{\mathcal{E}}(1)$ such that $\boldsymbol\rho=\widetilde\tau\rho$. Then, $\theta$ is the contact form corresponding to the CR scale $\widetilde\tau|_{\mathcal{N}}$.
By the same argument as in the proof of~\cite[Lemma~3.1]{Hirachi}, we can take $F\in\widetilde{\mathcal{E}}(0)$, $\boldsymbol{G}\in\widetilde{\mathcal{E}} (-n-1)$ which satisfy
\begin{gather*}
\widetilde{\Delta}\big(\log\widetilde\tau+F+\boldsymbol{G}\boldsymbol{\rho}^{n+1}\log\rho\big)=O\big(\rho^\infty\big), \qquad
F=O(\rho), \qquad \boldsymbol{G}|_\mathcal{N}=\frac{(-1)^n}{n!(n+1)!}Q.
\end{gather*}
We set $G:=\widetilde{\tau}^{n+1}\boldsymbol{G}\in\widetilde{\mathcal{E}}(0)$. By \eqref{laplacians} and the equation $\boldsymbol\rho\widetilde\Delta\log\boldsymbol\rho=n+1$, we have
\begin{gather*}
\Delta\big(\log\rho-F-G\rho^{n+1}\log\rho\big)=n+1+O(\rho^\infty).
\end{gather*}
Then, by using \eqref{Green2}, we compute as
\begin{gather*}
(n+1) \operatorname{lp}\int_{\rho>\epsilon} \operatorname{vol}_g
 =\operatorname{lp}\int_{\rho>\epsilon}\Delta\big(\log\rho-F-G\rho^{n+1}\log\rho\big)\operatorname{vol}_g \\
\hphantom{(n+1) \operatorname{lp}\int_{\rho>\epsilon} \operatorname{vol}_g}{}
=-(n!)^{-1}\operatorname{lp}\int_{\rho=\epsilon}\! N\big(\log\rho-F-G\rho^{n+1}\log\rho\big) \cdot (1+O(\epsilon))\epsilon^{-n} \vartheta\wedge({\rm d}\vartheta)^n \\
\hphantom{(n+1) \operatorname{lp}\int_{\rho>\epsilon} \operatorname{vol}_g}{}
=\frac{n+1}{n!}\int_M G\,\theta\wedge({\rm d}\theta)^n \\
\hphantom{(n+1) \operatorname{lp}\int_{\rho>\epsilon} \operatorname{vol}_g}{}
=\frac{(-1)^n}{(n!)^3} \overline{Q}.
\end{gather*}
Thus we complete the proof.
\end{proof}

\begin{proof}[Proof of Theorem \ref{van-total-CR-Q}] Let $\rho$ be an arbitrary def\/ining function of $M$, and $\widetilde\tau\in\widetilde{\mathcal{E}} (1)$ the density on $\overline X$ def\/ined by $\boldsymbol\rho=\widetilde\tau\rho$. Then $\alpha:=-i\partial\overline{\partial}\log\widetilde\tau$ is a closed $(1,1)$-form on $\overline X$. The volume form of $g$ is given by $\operatorname{vol}_g=\omega^{n+1}/(n+1)!$ with the fundamental 2-form $\omega=ig_{j\overline k}\theta^j\wedge\theta^{\overline k}$. Near the boundary $M$ in $X$, we have
\begin{gather*}
\omega=-i\partial\overline{\partial}\log\boldsymbol\rho=-i\partial\overline{\partial}\log\rho+\alpha.
\end{gather*}
Since the logarithmic term in the volume expansion is determined by the behavior of $\operatorname{vol}_g$ near the boundary, we compute as
\begin{gather*}
(n+1)! \operatorname{lp}\int_{\rho>\epsilon} \operatorname{vol}_g
 = \operatorname{lp}\int_{\rho>\epsilon} (-i\partial\overline{\partial}\log\rho+\alpha)^{n+1} \\
\hphantom{(n+1)! \operatorname{lp}\int_{\rho>\epsilon} \operatorname{vol}_g}{} =\operatorname{lp}\int_{\rho>\epsilon} \alpha^{n+1}+
\operatorname{lp}\int_{\rho>\epsilon} \sum_{k=1}^{n+1} \binom{n+1}{k} (-i\partial\overline{\partial}\log\rho)^{k}\wedge\alpha^{n+1-k}.
\end{gather*}
The f\/irst term in the last line is 0 since $\alpha$ is smooth up to the boundary. Using $-i\partial\overline{\partial}\log\rho={\rm d}(\vartheta/\rho)$ and $d\alpha=0$, we also have
\begin{gather*}
\operatorname{lp}\int_{\rho>\epsilon} (-i\partial\overline{\partial}\log\rho)^{k}\wedge\alpha^{n+1-k}=\operatorname{lp} \epsilon^{-k}\int_{\rho=\epsilon} \vartheta\wedge({\rm d}\vartheta)^{k-1}\wedge \alpha^{n+1-k} =0.
\end{gather*}
Thus, by Lemma \ref{Q-lp} we obtain $\overline Q=0$.
\end{proof}

\section{Proof of Theorem \ref{self-adjoint}}
We will recall the def\/initions of the $P^\prime$-operator and the $Q^\prime$-curvature. A CR scale $\tau\in\mathcal{E}(1)$ is called {\it pseudo-Einstein} if it has an extension $\widetilde\tau
\in\widetilde{\mathcal{E}}(1)$ such that $\partial\overline{\partial}\log\widetilde\tau=0$ near $\mathcal{N}$ in $\widetilde X$. The corresponding contact form $\theta$ is called a {\it pseudo-Einstein contact form} and characterized in terms of associated pseudo-hermitian structure; see \cite{Hirachi, Hirachi/Marugame/Matsumoto, Lee}. If $\tau$ is a pseudo-Einstein CR scale, another $\widehat\tau$ is pseudo-Einstein if and only if $\widehat\tau=e^{-\Upsilon}\tau$ for a CR pluriharmonic function $\Upsilon\in\mathcal{P}$. For any $f\in\mathcal{P}$, we take an extension
$\widetilde f\in\widetilde{\mathcal{E}}(0)$ such that $\partial\overline{\partial}\widetilde f=0$ near $M$ in $\overline X$ and def\/ine
\begin{gather*}
P^\prime f=-\widetilde\Delta^{n+1}\big(\widetilde f\log\widetilde\tau\big)|_\mathcal{N}\in\mathcal{E}(-n-1).
\end{gather*}
We note that the germs of $\widetilde\tau$ and $\widetilde f$ along $\mathcal{N}$ is unique, and $P^\prime f$ is assured to be a density by $\widetilde\Delta \widetilde f|_\mathcal{N}=0$. The $Q^\prime$-curvature is def\/ined by
\begin{gather*}
Q^\prime=\widetilde\Delta^{n+1}(\log\widetilde\tau)^2 |_\mathcal{N}\in\mathcal{E}(-n-1).
\end{gather*}
Here, the homogeneity of $Q^\prime$ follows from the fact $\widetilde\Delta\log\widetilde\tau\,|_\mathcal{N}=0$.

To prove the formal self-adjointness of $P^\prime$, we use its characterization in terms of the metric $g$. We def\/ine a dif\/ferential operator $\Delta^\prime$ by $\Delta^\prime f
=-g^{i\overline{j}}\partial_i\partial_{\overline{j}}f$. Since $g$ is K\"ahler near the boundary, $\Delta^\prime$ agrees with $\Delta$ near $M$ in $X$.

\begin{Lemma}[{\cite[Lemma 4.4]{Hirachi}}]\label{Delta}
Let $\tau\in\mathcal{E}(1)$ be a pseudo-Einstein CR scale and $\widetilde\tau\in\widetilde{\mathcal{E}}(1)$ its extension such that $\partial\overline{\partial}\log\widetilde\tau=0$ near $\mathcal{N}$ in $\widetilde X$. Let $\rho=\boldsymbol\rho/\widetilde\tau$ be the corresponding defining function. Then, for any $f\in C^\infty(\overline X)$ which is pluriharmonic in a neighborhood of $M$ in $\overline X$, there exist $F, G\in C^\infty(\overline X)$ such that $F=O(\rho)$ and
\begin{gather*}
\Delta^\prime \big(f\log\rho-F-G\rho^{n+1}\log\rho\big)=(n+1)f+O\big(\rho^\infty\big).
\end{gather*}
Moreover, ${\tau}^{-n-1}G|_M=\frac{(-1)^{n+1}}{(n+1)!n!}P^\prime f$ holds.
\end{Lemma}

In the statement of \cite[Lemma~4.4]{Hirachi}, the Laplacian $\Delta$ is used, but we may replace it by $\Delta^\prime$ since they agree near the boundary in $X$.

\begin{proof}[Proof of Theorem \ref{self-adjoint}]
We extend $f_j$ to a function on $\overline X$ such that $\partial\overline{\partial}f_j=0$ in a neighborhood of $M$ in $\overline X$. Let
$\tau$ be a pseudo-Einstein CR scale and $\rho=\boldsymbol\rho/\widetilde\tau$ the corresponding def\/ining function.
Then we have $\omega=-i\partial\overline{\partial}\log\rho$ near $M$ in $X$.
We take $F_j$, $G_j$ as in Lemma \ref{Delta} so that $u_j:=f_j\log\rho-F_j-G_j\rho^{n+1}\log\rho$ satisf\/ies $\Delta^\prime u_j=(n+1)f_j+O(\rho^\infty)$. We consider the coef\/f\/icient of $\log\epsilon$ in the expansion of the integral
\begin{gather*}
I_\epsilon=\operatorname{Re}\int_{\rho>\epsilon}\bigl(i\partial f_1\wedge \overline{\partial}u_2\wedge\omega^n+
i\partial f_2\wedge \overline{\partial}u_1\wedge\omega^n-f_1f_2\,\omega^{n+1}\bigr),
\end{gather*}
which is symmetric in the indices 1 and 2. Since $d\omega=0, \partial\overline{\partial}f_2=0$ near $M$ in $\overline X$, we have
\begin{gather*}
i\partial f_1 \wedge\overline{\partial}u_2\wedge\omega^n = {\rm d}\big(if_1\overline{\partial}u_2\wedge\omega^n\big) -if_1\partial\overline{\partial}u_2\wedge \omega^n+inf_1\overline{\partial}u_2\wedge {\rm d}\omega \wedge \omega^{n-1} \\
\hphantom{i\partial f_1 \wedge\overline{\partial}u_2\wedge\omega^n }{}
 ={\rm d}\big(if_1 \overline{\partial}u_2\wedge\omega^n\big)+\frac{1}{n+1}f_1\Delta^\prime u_2 \omega^{n+1}+({\rm cpt\ supp}), \\
i\partial f_2\wedge \overline{\partial}u_1\wedge\omega^n =-{\rm d}\big(iu_1\partial f_2\wedge\omega^n\big)+
({\rm cpt\ supp}),
\end{gather*}
where $({\rm cpt\ supp})$ stands for a compactly supported form on $X$. Thus,
\begin{gather*}
I_\epsilon=\int_{\rho>\epsilon}\frac{1}{n+1}f_1\bigl(\Delta^\prime u_2-(n+1)f_2\bigr)\omega^{n+1} \\
\hphantom{I_\epsilon=}{} +\operatorname{Re}\int_{\rho=\epsilon} i(f_1\overline{\partial}u_2-u_1\partial f_2)\wedge\omega^n
+\int_{\rho>\epsilon}({\rm cpt\ supp}).
\end{gather*}
The f\/irst and the third terms contain no log terms. Since $\omega={\rm d}(\vartheta/\rho)$ near $M$ in $X$, the second term is computed as
\begin{gather*}
\operatorname{Re}\int_{\rho=\epsilon} i(f_1\overline{\partial}u_2-u_1\partial f_2)\wedge\omega^n
 =\epsilon^{-n}\operatorname{Re}\int_{\rho=\epsilon} \Bigl(if_1\overline{\partial}\big(f_2\log\rho-F_2-G_2\rho^{n+1}\log\rho\big)
\wedge({\rm d}\vartheta)^n \\
 \qquad {} -i\big(f_1\log\rho-F_1-G_1\rho^{n+1}\log\rho\big)\wedge\partial f_2\wedge ({\rm d}\vartheta)^n\Bigr)+O\big(\epsilon^\infty\big).
\end{gather*}
The logarithmic term in the right-hand side is
\begin{gather*}
\log\epsilon\int_{\rho=\epsilon}(n+1)f_1G_2\vartheta\wedge({\rm d}\vartheta)^n+2\epsilon^{-n}\log\epsilon
 \operatorname{Re}\int_{\rho=\epsilon}if_1\overline{\partial}f_2\wedge({\rm d}\vartheta)^n+O(\epsilon\log\epsilon).
\end{gather*}
The coef\/f\/icient of $\log\epsilon$ in the f\/irst term is
\begin{gather}\label{P-int}
\frac{(-1)^{n+1}}{(n!)^2}\int_M f_1P^\prime f_2.
\end{gather}
The second term is equal to
\begin{gather*}
2\epsilon^{-n}\log\epsilon \operatorname{Re}\int_{\rho>\epsilon} i\partial f_1\wedge\overline{\partial}f_2\wedge({\rm d}\vartheta)^n
+\epsilon^{-n}\log\epsilon\int_{\rho>\epsilon} ({\rm cpt\ supp}).
\end{gather*}
The f\/irst term in this formula is symmetric in the indices 1 and 2 while the second term gives no $\log\epsilon$ term. Therefore, \eqref{P-int} should also be symmetric in~1 and~2, which implies the formal self-adjointness of~$P^\prime$.
\end{proof}

\section{Proof of Theorem \ref{CR-invariance}}
The formal self-adjointness of the $P^\prime$-operator implies the CR invariance of the total $Q^\prime$-curva\-tu\-re. When $n\ge 2$, the CR invariance can also be proved by the following characterization of $\overline{Q}^\prime$ in terms of the hermitian metric $g$ on $X$ whose fundamental 2-form $\omega=ig_{j\overline k}\theta^j\wedge\theta^{\overline k}$ agrees with $-i\partial\overline{\partial}\log\boldsymbol\rho$ near $M$ in $X$:

\begin{Theorem}[{\cite[Theorem 5.6]{Hirachi}}]\label{total-Q-prime}
Let $\tau\in\mathcal{E}(1)$ be a pseudo-Einstein CR scale and $\widetilde\tau\in\widetilde{\mathcal{E}}(1)$ its extension such that $\partial\overline{\partial}\log\widetilde\tau=0$ near $\mathcal{N}$ in $\widetilde X$. Let $\rho=\boldsymbol\rho/\widetilde\tau$ be the corresponding defining function. Then we have
\begin{gather}\label{Q-prime-lp}
\operatorname{lp}\int_{r>\epsilon} i\partial\log \rho\wedge \overline{\partial}\log\rho\wedge \omega^n
=\frac{(-1)^n}{2(n!)^2}\overline{Q}^\prime
\end{gather}
for any defining function $r$.
\end{Theorem}

In \cite[Theorem 5.6]{Hirachi}, it is assumed that $X$ is Stein and $\omega=-i\partial\overline{\partial}\log\rho$ globally on $X$, but as the logarithmic term is determined by the boundary behavior, it is suf\/f\/icient to assume $\omega=-i\partial\overline{\partial}\log\rho$ near $M$ in $X$ as above.

\begin{proof}[Proof of Theorem \ref{CR-invariance}] Let $\tau$, $\rho$ be as in Theorem~\ref{total-Q-prime} and let $\widehat\rho$ be the def\/ining function corresponding to another pseudo-Einstein CR scale $\widehat\tau$. Then we can write as $\widehat\rho=e^{\Upsilon}\rho$ with $\Upsilon\in C^\infty(\overline X)$ such that $\partial\overline{\partial}\Upsilon=0$ near~$M$ in~$\overline X$.

Using the def\/ining function $\rho$ for $r$ in the formula \eqref{Q-prime-lp}, we compute as
\begin{gather*}
\operatorname{lp}\int_{\rho>\epsilon} i\partial\log \widehat\rho\wedge \overline{\partial}\log\widehat\rho\wedge \omega^n=
\operatorname{lp}\int_{\rho>\epsilon} i(\partial\log \rho+\partial\Upsilon)\wedge (\overline{\partial}\log\rho+\overline{\partial}\Upsilon)\wedge \omega^n \\
\hphantom{\operatorname{lp}\int_{\rho>\epsilon} i\partial\log \widehat\rho\wedge \overline{\partial}\log\widehat\rho\wedge \omega^n}{}
 =\operatorname{lp}\int_{\rho>\epsilon} i\partial\log\rho\wedge \overline{\partial}\log\rho\wedge \omega^n +\operatorname{lp}\int_{\rho>\epsilon} i\partial\Upsilon\wedge \overline{\partial}\Upsilon\wedge \omega^n \\
\hphantom{\operatorname{lp}\int_{\rho>\epsilon} i\partial\log \widehat\rho\wedge \overline{\partial}\log\widehat\rho\wedge \omega^n=}{}
 +2\operatorname{Re}\ \operatorname{lp}\int_{\rho>\epsilon} i\partial\log\rho\wedge \overline{\partial}\Upsilon\wedge \omega^n.
\end{gather*}
The second term in the last line is
\begin{gather*}
\operatorname{lp}\int_{\rho>\epsilon} i\partial\Upsilon\wedge \overline{\partial}\Upsilon\wedge \omega^n=\operatorname{lp}\int_{\rho=\epsilon} i\Upsilon\overline{\partial}\Upsilon\wedge \omega^n+\operatorname{lp}\int_{\rho>\epsilon}({\rm cpt\ supp})=0.
\end{gather*}
Since $\omega={\rm d}(\vartheta/\rho)$ near~$M$ in~$X$, we have
\begin{gather*}
\int_{\rho>\epsilon} i\partial\log\rho\wedge \overline{\partial}\Upsilon\wedge \omega^n =\log\epsilon
\int_{\rho=\epsilon} i\overline{\partial}\Upsilon\wedge \omega^n+\int_{\rho>\epsilon} ({\rm cpt\ supp}) \\
\hphantom{\int_{\rho>\epsilon} i\partial\log\rho\wedge \overline{\partial}\Upsilon\wedge \omega^n}{}
 =\epsilon^{-n}\log\epsilon\int_{\rho=\epsilon} i\overline{\partial}\Upsilon\wedge (d\vartheta)^n+\int_{\rho>\epsilon} ({\rm cpt\ supp}) \\
 \hphantom{\int_{\rho>\epsilon} i\partial\log\rho\wedge \overline{\partial}\Upsilon\wedge \omega^n}{}
 =\epsilon^{-n}\log\epsilon\int_{\rho>\epsilon} ({\rm cpt\ supp})+\int_{\rho>\epsilon} ({\rm cpt\ supp}),
\end{gather*}
which implies that the third term is also~0. Thus, $\overline{Q}^\prime$ is independent of the choice of a pseudo-Einstein CR scale~$\tau$.\end{proof}

\subsection*{Acknowledgements}

The author would like to thank the referees for their comments which were helpful for the improvement of the manuscript.

\pdfbookmark[1]{References}{ref}
\LastPageEnding

\end{document}